\documentclass[a4paper,leqno,11pt]{amsart}

\usepackage{amsfonts,amssymb,verbatim,amsmath,amsthm,latexsym,textcomp,amscd}
\usepackage{latexsym,amsfonts,amssymb,epsfig,verbatim}
\usepackage{amsmath,amsthm,amssymb,latexsym,graphics,textcomp}
\usepackage{graphicx}
\usepackage{color}
\usepackage{url}
\usepackage{enumerate}
\usepackage[mathscr]{euscript}
\usepackage{xcolor}

\input xy
\xyoption{all}

\usepackage{hyperref}

\usepackage{a4}
\theoremstyle{plain}
\newtheorem{theorem}{Theorem}[section]
\newtheorem{thm}[theorem]{Theorem}
\newtheorem{lemma}[theorem]{Lemma}
\newtheorem{cor}[theorem]{Corollary}
\newtheorem{prop}[theorem]{Proposition}

\newtheorem{remark}[theorem]{Remark}

\theoremstyle{definition}
\newtheorem{defn}[theorem]{Definition}
\newtheorem{rmk}[theorem]{Remark}
\newtheorem{exam}[theorem]{Example}

\newcommand{\mM}{\underline{M}}
\newcommand{\uH}{\underline{H}}
\newcommand{\mN}{\underline{N}}
\newcommand{\uM}{\underline{M}}

\newcommand{\uA}{\underline{A}}

\newcommand{\Z}{\mathbb Z}

\newcommand{\bZp}{\langle \Z/p \rangle}

\newcommand{\bZq}{\langle \Z/q \rangle}

\newcommand{\CC}{\mathcal{C}}

\newcommand{\KK}{\mathcal{K}}

\newcommand{\uZ}{\underline{\mathbb{Z}}}

\author{Surojit Ghosh}

\email{surojitghosh89@gmail.com; gsurojit@campus.haifa.ac.il}
\address{Department of Mathematics,
University of Haifa,
Haifa-3498838, Israel.} 


\begin{document}
\title{Structure of $C_{pq}$-cohomology of points and slice tower}
\maketitle

\begin{abstract}
The additive structure of the $RO(C_{pq})$-graded Bredon cohomology $S^0$ with coefficients in the constant Mackey functor was computed in \cite{BG19}. Using that computation, the ring structure in the positive degrees has been computed here. Further, we calculate the slices of the spectrum $S^V \wedge H\uZ$ for any representation $V.$ 
\end{abstract}

\section{Introduction}
The $RO(G)$-graded Bredon cohomology of $S^0$ is hard to compute. The Mackey functors $\uH^\ast_G(S^0; \uA)$ were computed by Stong for the group $C_2$ and by Lewis and Stong for the groups $C_p$ (\cite{Lew88}). The multiplicative structure was also computed by Lewis in the same paper. With constant coefficients $\underline{\Z/p}$, the computations of the cohomology of $C_p$-orbits were performed by Stong (Appendix of \cite{Car00}). In \cite{BG19}, the authors carry the computations of $\uH^\ast_G(S^0; \uM)$ further to $G=C_{pq}$ where $p$ and $q$ are distinct odd primes, with coefficients in either Burnside ring Mackey functor or in constant Mackey functor. But they avoid computations of the ring structure. In this article, using the computations in \cite{BG19}, we compute the multiplicative structure of $\tilde{H}^{\bigstar}_{C_{pq}}(S^0)$ in the positive degree ($\bigstar= V -\ast)$ with coefficients in the constant Mackey functor. and obtain ({\it cf.} Theorem \ref{main1})

\begin{thm}
In positive degree, the ring $\tilde{H}^{\bigstar}_{C_{pq}}(S^0)$ is isomorphic to 
$$ \frac{\Z[a_\xi, a_{\xi^p}, a_{\xi^q}, u_\xi, u_{\xi^p}, u_{\xi^q}]}{(pqa_\xi, q a_{\xi^p }, pa_{\xi^q}, u_\xi a_{\xi^p}- p u_{\xi^p} a_{\xi}, u_\xi a_{\xi^q}-q u_{\xi^q} a_{\xi})}$$ where $a_V$ and $u_V$ are certain classes defined in \ref{au}.
\end{thm}

Next, we study the Slice tower associated to certain $G$-spectra. The slice spectral sequence is a modern tool in computational equivariant homotopy theory. It is an analog of the motivic spectral sequence developed by Voevodsky \cite{Voe02}. This is a homotopy spectral sequence on a certain tower of fibration, called the slice tower.  The slice tower is an equivariant analog of the Postnikov tower from  classical stable homotopy theory. The slice tower for finite groups was first developed in \cite{HHR16} to solve the Kervaire invariant one problem. It is actually build on the work of Dugger for $C_2$-spectra \cite{Dug05}. This filtration has been reformulated in Ullman's work in \cite{Ull13}. In this article, we rely on the Ullman's new structure on slice filtration. Let $Sp^G$ be the category of genuine $G$-spectra for the finite group $G.$ The localizing subcategory of $Sp^G$ are those closed under weak equivalence, cofibration, extensions, and coproducts. An $n$-slice category is then the localizing subcategory of $Sp^G$ generated by the $G$-spectra of the form $\Sigma_G^{\infty}G/H_+ \wedge S^{k \rho_H},$ where $H$ is a subgroup of $G$ and $\rho_H$ is the regular representation of $H.$

For a spectrum $X$, we write $X \geq n$ if it is in the $n$-slice category. Associated to the $n$-slice category there is a natural localization functor $P^{n-1}.$ Note that the $(n+1)$-slice category sits insides the $n$-slice category. Therefore, there is a natural transformation $P^n \to P^{n-1}$, which yields the slice tower for each spectrum $X$ as follows:

$$\cdots \to P^{n+1}X \to P^nX \to P^{n-1}X \to \cdots$$ 
The homotopy fibre at each level $P^n X \to P^{n-1}X$ is called the $n$-slice for the spectrum $X$ and is denoted by $P^n_nX.$

 There is an equivariant analog of the Eilenberg MacLane spectrum $H \uM$ for each Mackey functor $\uM.$ This spectrum is always $0$-slice and hence has the trivial slice tower. Next, the natural question is how to determine the slices for any module over Eilenberg Maclane spectrum. The computation of the slice tower for $S^n \wedge H\uZ$ was made in \cite{Yar17} and for $S^{n\xi}\wedge H\uZ$ in \cite{HHR2} for the group $C_{p^n}.$ Recently, in \cite{BY18} the authors compute the
  $C_2 \times C_2$-slice tower for the spectrum $S^n \wedge H\underline{\Z}/2.$ 
  
  In this article, we investigate the slices of the spectrum $\Sigma^V  H\uZ$ for $C_{pq}$ where $V$ is any virtual representation of $C_{pq}$ and obtain( {\it cf.} Theorem \ref{main2})
  
\begin{thm}
Let $\alpha \in RO(C_{pq}),$ then $S^\alpha \wedge H\uZ$ has a $dim(\alpha)$-slice $S^\beta \wedge H\uZ$ for some $\beta \in RO(C_{pq}).$ The other slices are suspensions of $H \KK_p \bZp$ or $H\KK_q \bZq$ or wedges of them.
\end{thm}
  
\subsection{Organization.} In section 2, we recall some definitions and results from $RO(G)$-graded Bredon cohomology theory. We use these  to write the multiplicative structure of the cohomology of $S^0$ in positive degree in section $3$. In section $4$, we describe the slice tower for certain spectrum $S^V \wedge H\uZ.$

\section{preliminaries}

We recall certain basic ideas and techniques in Bredon cohomology, and fix the notations used throughout the paper. Details and proofs may be found in \cite{May96}. The notation $G$ will be used for the cyclic group $C_{pq}$ of order $pq$ where $p < q$ are two distinct primes, though most of the facts in this section also hold for a general finite group.  

For a unitary representation $V$ of the group $G,$ let $S^V$ be the one point compactification of $V,$ with the action induced from that of $V$.
 
 Equivariant homotopy and cohomology theories are more naturally graded by $RO(G),$ the Grothendieck group of finite real orthogonal representations of $G.$ To obtain this kind of theory one needs more structure on the coefficients. These are called Mackey functors.

\begin{defn}
A Mackey functor consists of a pair $\uM = (\uM_\ast , \uM^\ast )$ of functors from the category of finite $G$-sets to $\mathcal{A} b$, with $\uM_\ast$ covariant and $\uM^\ast$ contravariant. On every object $S$, $\uM^\ast$ and $\uM_\ast$ have the same value which we denote by $\uM(S)$,  and $\uM$ carries disjoint unions to direct sums. The functors are required to satisfy that for every pullback diagram of finite $G$-sets as below 
$$\xymatrix{ P \ar[r]^\delta    \ar[d]^\gamma                             & X \ar[d]^\alpha \\ 
                        Y \ar[r]^\beta                                                     &  Z,}$$
one has $\uM^\ast(\alpha) \circ \uM_\ast(\beta) = \uM_\ast(\delta) \circ \uM^\ast(\gamma).$ 
\end{defn}

\begin{exam}For an Abelian group $C,$ an easy example for a Mackey functor is the constant Mackey functor $\underline{C}$ defined by the assignment $\underline{C}(S) = Map^G(S, C),$ the set of $G$-maps from the $G$-orbit $S$ to $C$ with trivial $G$-action.
\end{exam}

Equivariant cohomology theories are represented by $G$-spectra. Naive $G$-spectra are those in which only desuspension with respect to trivial $G$-spheres is allowed. Usually, what we mean by $G$-spectra are those in which desuspension with respect to all representation spheres are allowed. In the viewpoint of \cite{LMS86}, naive $G$-spectra are indexed over a trivial $G$-universe and $G$-spectra are indexed over a complete $G$-universe. As we are allowed to take desuspension with respect to representation-spheres, the associated cohomology theories become $RO(G)$-graded. 

We consider orthogonal $G$-spectra with positive complete model structure to model the equivariant stable homotopy theory, which can be read off from \cite[Appendix A, B]{HHR16}. In particular, we use $(Sp^G, \wedge, \mathbb{S}^0)$ for the symmetric model category of orthogonal $G$-spectra. 

Every $G$-set $S$ yields a suspension spectrum $\Sigma^\infty_G S_+$ in the category of $G$-spectra. It turns out that the category of finite $G$-sets, homotopy classes of spectrum maps as morphisms, is naturally isomorphic to the Burnside category. Thus, the homotopy groups of $G$-spectra are naturally Mackey functors. For an orthogonal spectrum $X,$  denote by $\pi_{\bigstar}(X)$ its $RO(G)$-graded homotopy groups. In particular, for $\alpha = V -W \in RO(G),$ 
$$\pi_\alpha(X) = [S^V, S^W \wedge X]^G.$$

In non-equivariant homotopy theory, for each Abelian group $A$ there is an Eilenberg-MacLane spectra $HA$ satisfying  $$\pi_n(HA) = \begin{cases} A, & \text{if }n =0\\ 
0, & \text{otherwise.}
 \end{cases}$$
 
In the category of orthogonal spectra, we also have Eilenberg-MacLane spectra for each Mackey functor.

\begin{prop}
Let $\uM$ be a $G$-Mackey functor. Then there exist an equivariant Eilenberg-MacLane spectrum $H\uM$, unique up to homotopy in $Sp^G.$ 
\end{prop}
\begin{proof}
See \cite[Theorem 5.3]{GM95a}.
\end{proof} 

Thus, equivariant Eilenberg-MacLane spectra arise from Mackey functors. This is a theorem of Lewis, May, and McClure in Chapter XIII of \cite{May96}.  Therefore,  integer-graded cohomology associated to a coefficient system extends to $RO(G)$-graded cohomology theory if and only if the coefficient system has an underlying Mackey functor structure.

\begin{defn}
 A $RO(G)$-graded cohomology theory consists, for each $\alpha \in RO(G),$ of functors $E^\alpha$ from reduced equivariant CW complexes to Abelian groups which satisfy the usual axioms - homotopy invariance, excision, long exact sequence and the wedge axiom. 
\end{defn}
It is interesting to note that the suspension isomorphism for $RO(G)$-graded cohomology theories takes the form 
$E^{\alpha}(X) \cong E^{\alpha + V}(S^V \wedge X)$ for every based $G$-space $X$ and representation $V$.

We recall that there are change of groups functors on equivariant spectra. The restriction functor from $G$-spectra to $H$-spectra has a left adjoint given by smashing with $(G/H)_+$. This also induces an isomorphism for cohomology with Mackey functor coefficients 
$$\tilde{H}^\alpha_G((G/H)_+\wedge X ; \uM)\cong \tilde{H}^\alpha_H(X; res_H(\uM))$$

The $RO(G)$-graded theories may also be assumed to be Mackey functor-valued as in the definition below.   
\begin{defn}
Let $X$ be a pointed $G$-space, $\uM$ be any Mackey functor, $\alpha \in RO(G)$. Then the Mackey functor valued cohomology $\uH^{\alpha}_{G}(X;\uM)$ is defined: 
$$\uH^{\alpha}_{G}(X;\uM)(G/K) = \tilde{H}^{\alpha}_{G}({(G/K)}_+ \wedge X;\uM).$$
The restriction and transfer maps are induced by the appropriate maps of $G$-spectra. 
\end{defn}

\section{The ring structure}
One natural question is: {\it What is the structure on a Mackey functor which induces a ring structure on the cohomology of spaces?} There is a box product $\Box$ on the category of Mackey functors. For two Mackey functors $\mM, \mN$, this is obtained by taking the left Kan extension along 
$$\xymatrix{ Burn_G \times Burn_G \ar[r] \ar[d] & Ab \\ 
 Burn_G \ar@{-->}[ru]}$$
The right arrow in the top row is given by $(S,T)\mapsto \mM(S)\times \mN(T)$. The left vertical arrow is given by $(S,T) \mapsto S\sqcup T$. Mackey functors inducing ring valued cohomology theories are monoids under the box product $\Box$. The constant Mackey functor $\uZ$ is a monoid under $\Box$. Therefore $\tilde{H}^\bigstar_G(S^0;\uZ)$ has a graded ring structure. In this section, we try to understand the ring structure for the group $C_{pq}.$

We start by recalling some $C_p$-Mackey functors from \cite{Lew88}. In the diagrams below the downward arrows are restrictions and the upward arrows are the transfers.

$$\xymatrix{   & \Z  \ar@/_.5pc/[dd]_{p} &&  & \Z/p \ar@/_.5pc/[dd]  &&  & \Z  \ar@/_.5pc/[dd]_{1}\\ 
L_p :                                                     &&& \langle \Z/p \rangle_p :  &&& R_p:     \\
 & \Z \ar@/_.5pc/[uu]_{1}                    &&   & 0 \ar@/_.5pc/[uu]    &&   & \Z \ar@/_.5pc/[uu] _{p} }$$

Now observe that the Burnside category  $Burn_{C_{pq}}$ is isomorphic to  $Burn_{C_{p}}\otimes Burn_{C_q}$ formed as the product set of objects and tensor product set of morphisms. Thus we may define a $C_{pq}$-Mackey functor by tensoring Mackey functors on $C_p$ and $C_q$. The following Mackey functors have special importance in our case.

\begin{defn}
For a $C_p$-Mackey functor $\uM$, define $C_{pq}$-Mackey functors

$$\CC_p\uM: = \uM \otimes L_q, \,  \KK_p\uM : = \uM \otimes R_q.$$
\end{defn}
Also we denote  
$$ R_{pq} := R_p \otimes R_q, \, L_{pq} := L_p \otimes L_q. $$

Note that the Mackey functor $R_{pq}$ is the constant Mackey functor for the group $C_{pq}.$ Therefore, depending on the context, sometimes we use $\uZ$ instead of $R_{pq}.$

We call a Mackey functor cohomological if it is  a module over $\uZ.$ In other words, a Mackey functor $\uM$ is cohomological if the composite $tr^H_K res^H_K$ is simply multiplication by $|H|/|K|$ for all $K\leq H \leq G.$ A cohomological Mackey functor for the group $C_{pq}$ satisfies the following:  

\begin{prop}\label{cohmac}
Let $\uM$ be a $C_{pq}$-cohomological Mackey functor such that both the groups $\uM(C_{pq}/C_p)$ and $\uM(C_{pq}/C_q)$ vanish. Then, the Mackey functor $\uM$ is trivial.
\end{prop}

\begin{proof}
Since $\uM$ is cohomological, for $K \leq H \leq G,$ the composition 

$$tr^H_K res^H_K: \uM(G/H) \to \uM(G/H)$$ is given by the multiplication by the index $|H|/|K|.$ Now pick an element $x \in \uM(C_{pq}/C_{pq})$; applying above map we get 

$$px=0 \text{ and } qx=0.$$
Since, $p$ and $q$ are relatively prime, $x=0.$ The result follows.
\end{proof} 

 With the above notations, we recall from \cite{BG19} the additive structure of the Mackey functor $\uH^\alpha_G(S^0;R_{pq})$ as follows:
\begin{thm}\label{compr}Let $\alpha \in RO(C_{pq}).$ Then the Mackey functor
$$\uH^\alpha_G(S^0;R_{pq}) \cong \begin{cases}
\KK_p\bZp  & \mbox{if}~|\alpha|<0, |\alpha^{C_p}|>1, |\alpha^{C_q}|\leq 1~\mbox{odd} \\ 
\KK_q\bZq  & \mbox{if}~|\alpha|<0, |\alpha^{C_p}|\leq 1, |\alpha^{C_q}|> 1~\mbox{odd} \\
\KK_p\bZp \oplus \KK_q\bZq & \mbox{if}~|\alpha|<0, |\alpha^{C_p}|>1, |\alpha^{C_q}|>1~\mbox{odd} \\
   \KK_p\bZp \oplus \KK_q\bZq & \mbox{if}~|\alpha|>0, |\alpha^{C_p}|\leq 0, |\alpha^{C_q}|\leq 0~ \mbox{even} \\
\KK_p\bZp  & \mbox{if}~|\alpha|>0, |\alpha^{C_p}|\leq 0, |\alpha^{C_q}|> 0~ \mbox{even} \\
\KK_q\bZq  & \mbox{if}~|\alpha|>0, |\alpha^{C_p}|> 0, |\alpha^{C_q}|\leq 0~ \mbox{even} \\
R_{pq}  & \mbox{if}~|\alpha|=0, |\alpha^{C_p}|\leq 0, |\alpha^{C_q}|\leq 0 \\
L_{pq}  & \mbox{if}~|\alpha|=0, |\alpha^{C_p}|> 0, |\alpha^{C_q}|> 0 \\
\KK_pL_p  & \mbox{if}~|\alpha|=0, |\alpha^{C_p}|> 0, |\alpha^{C_q}|\leq 0 \\
\KK_qL_q  & \mbox{if}~|\alpha|=0, |\alpha^{C_p}|\leq 0, |\alpha^{C_q}|> 0 \\
0  &\mbox{otherwise}. 
\end{cases} $$
\end{thm}

Also, recall the definitions of the classes $a_V$ and $u_V$ in $\uH^\bigstar_G(S^0; \uZ)$ from \cite{HHR16}
\begin{defn}\label{au} 
(1) For a representation $V$ with $V^G =0,$ let $a_V \in \underline{\pi}_{-V}(S^0)$ be the map $S^0 \to S^V$ which embeds $S^0$ to $S^V$ to $0$ and $\infty$ in $S^V.$ We will also use $a_V$ for its Hurewicz image in $\underline{\pi}_{-V}(H \uZ)(G/G) \cong \uH^V_G(S^0; \uZ)(G/G).$

(2) For an orientable representation $W$ of dimension $n,$ let $u_W$ be the generator of $\underline{H}_n (S^W; \underline{\mathbb{Z}})(G/G)$ which restricts to the choice of orientation in  $ \underline{H}_n (S^W; \uZ)(G/e)\cong H_n(S^n; \mathbb{Z}).$ In cohomology grading , $u_W \in \tilde{H}^{V-dim V}_G(S^0; \uZ).$

In particular, for the group $G=C_{pq},$ where $p$ and $q$ are distinct odd primes, we have $a_{\xi^j} \in \tilde{H}^{\xi^j}_G(S^0; \uZ)$ and $u_{\xi^j}\in \tilde{H}^{\xi^j -2}_G(S^0; \uZ).$
\end{defn}

Next, we have equivalences of spectra which simplify computations:
 
\begin{lemma}\label{HK} There are the following equivalences

(1) $S^\xi \wedge H\uZ \cong S^{\xi^j} \wedge H \uZ$ if $(j, pq) =1.$

(2) $S^{\xi^p} \wedge H\uZ \cong S^{\xi^{jp}} \wedge H \uZ$ if $(j, q) =1.$

(3) $S^{\xi^q} \wedge H\uZ \cong S^{\xi^{jq}} \wedge H \uZ$ if $(j, p) =1.$

(4) $\Sigma^{\xi^p}H\uZ \simeq \Sigma^{\xi} H(\KK_p L_p).$

Moreover, there is a map $H\uZ  \to \Sigma^{\xi-\xi^p}H\uZ$
\end{lemma}
\begin{proof}
All these equivalences are directly follows from the computation of the homotopy groups of $S^{\xi^j -\xi^k}\wedge H \uZ.$ Therefore, we use the Theorem \ref{compr} to conclude the results.

Consider the Mackey functor map
$$L_p \to R_p$$ which is identity on the orbit $C_p/e.$ This map and the equivalence (4) give the following 

$$H\uZ \stackrel{\cong}{\to} \Sigma^{\xi-\xi^p}H\KK_p L_p \to \Sigma^{\xi-\xi^p}H \uZ.$$ We denote this map by $u_{\xi-\xi^p}.$
\end{proof}

Thus, only the $C_{pq}$- representations of the form $$V = a + b\xi +c \xi^p +d \xi^q \text{ for } a,b ,c ,d \in \Z$$ are useful for our computations. Hence, we are interested in the classes $a_\xi, a_{\xi^p}, a_{\xi^q}$ and $u_\xi, u_{\xi^p}, u_{\xi^q}.$ These classes satisfy the following

\begin{prop}\label{relation}
(1) For two fixed point free representations $V$ and $W$, we have $$a_{V+W} = a_Va_W \;\;\text{and} \; \; u_{V+W} = u_V u_W.$$

(2) $pqa_\xi=0,  q a_{\xi^p }=0,$ and $pa_{\xi^q} =0.$

(3) $u_\xi a_{\xi^p} = p u_{\xi^p} a_{\xi} $ and $u_\xi a_{\xi^q} = q u_{\xi^q} a_{\xi}$
\end{prop}

\begin{proof}
(1) is an general fact about any finite group, which follows from \cite{HHR16}.

 For (2), the values of $\underline{H}^{\xi}_{C_{pq}}(S^0; \uZ)$, $\underline{H}^{\xi^p}_{C_{pq}}(S^0; \uZ)$, and $\underline{H}^{\xi^q}_{C_{pq}}(S^0; \uZ)$ from the Theorem \ref{compr} yield the result.

For (3): We know there is a canonical map $S^\xi \to S^{\xi^p}$ with non-equivariant degree $p.$  We denote this map by $a_{\xi^p}/ a_{\xi}.$ We have a commutative diagram $$\xymatrix{ & S^{\xi}  \ar[dd]^{a_{\xi^p}/ a_{\xi}} \\ S^0 \ar[ur]^{a_{\xi}} \ar[dr]_{a_{\xi^p}} \\ & S^{\xi^p}}$$ This yields a map $a_{\xi^p}/ a_{\xi} \wedge H\uZ : S^\xi \wedge H\uZ \to S^{\xi^p} \wedge H\uZ.$ Now, we construct another map $u_{\xi}/ u_{\xi^p}$ such that $$\xymatrix{ & S^{\xi^p} \wedge H\uZ \ar[dd]^{u_{\xi}/ u_{\xi^p}} \\ S^2 \ar[ur]^{u_{\xi^p}} \ar[dr]_{u_{\xi}} \\ & S^{\xi}\wedge H\uZ}$$

The Lemma \ref{HK} yields a map $S^{\xi^p} \wedge H\uZ \to S^\xi \wedge H\uZ,$ of degree one. Therefore, the composition $(u_{\xi}/ u_{\xi^p}) (a_{\xi^p}/ a_{\xi})$ gives a factorization of the degree $p$ map on $S^\xi \wedge H\uZ.$ Hence the first part follows. The result for the prime $q$ follows analogously.
\end{proof}

Proposition \ref{relation} yields a map $$\Phi: \frac{\Z[a_\xi, a_{\xi^p}, a_{\xi^q}, u_\xi, u_{\xi^p}, u_{\xi^q}]}{(pqa_\xi, q a_{\xi^p }, pa_{\xi^q}, u_\xi a_{\xi^p}- p u_{\xi^p} a_{\xi}, u_\xi a_{\xi^q}-q u_{\xi^q} a_{\xi})} \longrightarrow \pi_\bigstar H\uZ.$$ 

The main result of this section is:

\begin{thm}\label{main1}
If $\bigstar = \ast - V,$  then the map $\Phi$ is a isomorphism of rings.
\end{thm}

\begin{proof}
Lemma \ref{HK} allows us to take $V$ to be the representation $k+ m \xi + n\xi^p + l \xi^q$ for non-negative integers $m, n$ and $l$. If $|\bigstar|<0$ or $>0$ odd, then by Theorem \ref{compr}, $\uH^\bigstar_G(S^0; \uZ)=0.$ Also, since all the classes $a_\xi^j$ and $u_\xi^j$ have even dimensions greater than equal to zero, the map $\Phi$ is an isomorphism if $|\bigstar| <0.$

We are left with the case $|\bigstar|$ is a  non-negative and even. Therefore, the dimension of the fixed points of $\bigstar$ are even. So $\bigstar$ is of the form $m \xi + n\xi^p + l \xi^q -2a$ for $m,n,l \geq 0$ and $a \in \Z.$ We denote this element by $(m,n,l,a).$

If $|\bigstar|=0$, $\tilde{H}^\bigstar_G(S^0; \uZ)$ is generated by the unique monomial $u_
\xi^m u_{\xi^p}^n u_{\xi^q}^l.$ 

Next, assume that $|\bigstar| >0.$ Then we may write $\bigstar = |u_\xi^{a_1}a_\xi^{b_1}u_{\xi^p}^{a_2}a_{\xi^p}^{b_2} u_{\xi^q}^{a_3}a_{\xi^q}^{b_3}|.$ From Proposition \ref{relation} (2), it is impossible to consider the presence of both the classes $a_{\xi^p}$ and $a_{\xi^q}$ for the non-trivial generator $u_\xi^{a_1}a_\xi^{b_1}u_{\xi^p}^{a_2}a_{\xi^p}^{b_2} u_{\xi^q}^{a_3}a_{\xi^q}^{b_3}.$ Without loss of generality we take $b_3=0.$ Again, using (3) of Proposition \ref{relation}, if $a_1 >0$ then either $b_2=0$ or $b_3=0.$ Then, for the $\bigstar = (m,n,l,a)$ we have at most three types of possible generators:

$$ {\it Type \: I}: \; u_\xi^{a-n-l}a_\xi^{m-a+n+1}u_{\xi^p}^n u_{\xi^q}^l$$ 

$${\it Type \: II}: \; a_\xi^m u_{\xi^p}^{a-l} a_{\xi^p}^{n-a+l} u_{\xi^q}^l$$

$${\it Type \: III}: \; a_\xi^m u_{\xi^p}^{n} u_{\xi^q}^{a-n} a_{\xi^q}^{l-a+n}$$

Here, the quadruple $(m,n,l,a)$ satisfies the following relations for different types:

$${\it Type \: I}: \; a \geq l+n, \; m \geq a-(l+n).$$

$${\it Type \: II}: \; a \geq l, \; n \geq a-l.$$

$${\it Type \: III}: \; a \geq n, \; l \geq a-n.$$
Suppose, {\it Type I} and {\it Type II} happen together, then we have $a=l+n$ and $m \geq 0.$ So, the generators in {\it Type I} and {\it Type II} yield only one term, $a_\xi^m u_{\xi^p}^n u_{\xi^q}^l,$ which maps non-trivially. 

If {\it Type I}, {\it Type II} and {\it Type III}  happen together, then we have again $a=l+n$. Then, for each $(m,n,l,a)$, we get only one term, $a_\xi^m u_{\xi^p}^n u_{\xi^q}^l,$ which maps non-trivially. Therefore, for each $\bigstar$ we are getting exactly one class in the domain of $\Phi.$ Hence, it is a bijection.
\end{proof}

\section{$C_{pq}$-slices for the spectrum $S^V \wedge H \uZ$}

In this section, for each $\alpha \in RO(C_{pq})$, we compute the $C_{pq}$-slices for the spectrum $\Sigma^\alpha H\uZ$. First we recall some useful facts about slices.

In the introduction we have already seen the meaning of $X \geq n$ for a $G$-spectrum $X.$ Such spectra are called $n$-slice connective.  Now, we define the $n$-slice coconnnective case. We denote it by $X \leq n.$

 \begin{defn}
A spectrum $X \leq n$ if the restriction $res^G_H(X) \leq n$ and $[S^{k\rho_G}, X]^G =0$ for all $k \geq 0$ such that $k dim(\rho_G)> n.$
\end{defn}

 By \cite[Corollary 2.6]{Hil12}, if a spectrum satisfies $k \leq X \leq n$,  the same holds for   all restrictions. But the definition of $n$-slice connective case is not computational. A more convenient criterion for $X \geq n$ is given by \cite[Theorem 2.10]{HY18}:
 
 \begin{thm}\label{connective}
 A $G$-spectrum $X \geq n$ if and only if the non equivariant homotopy groups $$\pi_k\Phi^H(X)=0$$ for all $H \leq G$ and $k < \frac{n}{|H|}.$ Here $\Phi^H(X)$ is the $H$-geometric fixed point of the spectrum $X$, that is, the geometric fixed point of $res^G_H(X).$
  \end{thm} 
  
\begin{rmk}  Note that the geometric fixed point spectrum $\Phi^{C_{pq}}(X)$ is trivial.
  \end{rmk}

\begin{prop}\label{regular}
If $X$ be a $m$-slice,  $\Sigma^{k \rho_G} X$ is  $(m + k|G|)$-slice for all $k \in \Z,$ that is,

$$P^{m+k|G|}_{m+k|G|}(\Sigma^{\rho_G} X) \simeq \Sigma^{\rho_G} P^m_m(X).$$
\end{prop}

\begin{proof}
See \cite[Corollary 4.25]{HHR16}.
\end{proof}

As we discussed that the Slice spectral sequence is a homotopy spectral sequence and the filtration for the homotopy groups of spectrum  $X$ is given by 

$$F^s\underline{\pi}_t(X) = ker(\underline{\pi}_t(X)\to \underline{\pi}_t(P^{t+s-1}X )$$ such that 
$$\underline{\pi}_{t-s}(P^t_t X) \cong F^s\underline{\pi}_{t-s}(X)/ F^{s+1}\underline{\pi}_{t-s}(X).$$

In \cite{Ull13}, it has been proved that the above filtration of the homotopy groups can be computed from the certain filtration of the Mackey functor in the following sense:

\begin{prop}\label{algfil}
If $n>0,$ then for a $(n+1)$-coconnected spectrum $X,$ then

$$F^s\underline{\pi}_t(X) \cong \mathcal{F}^{(s+n-1)/n}\underline{\pi}_n X.$$ Here, for a Mackey functor $\uM$ the filtration $\mathcal{F}^k\uM$ is given by as follows:

$$\mathcal{F}^k\uM(G/H) = \{x\in \uM(G/H): i^\ast_{|J|}x =0, \text{ for all } J \subset H, |J|\leq k \}$$
 and $i^\ast$ is a functor between the category of $G$-Mackey functors such that $$i^\ast_a\uM(G/H) = \begin{cases}
0, & \text{ if } |H|>a,\\
\uM(G/H), & \text{ otherwise.} 
 \end{cases}$$
\end{prop}
The restrictions and transfers are induced from $\uM.$
\begin{proof}
See \cite[Corollary 8.6]{Ull13} for the proof.
\end{proof}

By repeated application of the map $u_{(\xi-\xi^p)}$ in Lemma \ref{HK} and then smashing with the sphere $S^V$ we obtain a map $S^{V} \wedge H\uZ \to S^{V+l(\xi-\xi^p)}\wedge H\uZ$, denoted by $u_{l(\xi-\xi^p)}.$ Its cofiber is denoted by $Cof(u_{l(\xi-\xi^p)}).$ Then we obtain

\begin{prop}\label{cofiber}
The spectrum $Cof(u_{l(\xi-\xi^p)})$ is a wedge of suspensions of $H\KK_p\bZp$. Therefore, it has only $np$-slices for integer $n$.
\end{prop}

\begin{proof}
Note that the restriction of the cofiber to the subgroup $C_q$ is trivial. Since the cofiber is again $H\uZ$-module, it is cohomological. Using Proposition \ref{cohmac}, we conclude that $\underline{\pi}_n Cof(u_{l(\xi-\xi^p)})$ is non-zero if and only if $\underline{\pi}_n res_{C_p}(Cof(u_{l(\xi-\xi^p)}))$ is non-zero. Now, set $V = a+b \xi +c \xi^p + d \xi^q$ with $a, b,c$ and $d$ are non-negative integers. Therefore, $\underline{\pi}_n(Cof(u_{l(\xi-\xi^p)}))$ is non-trivial if and only if $n \in \{ a+2c-2, \cdots, a+2c-2l\}$ and for each such $n$, the homotopy group Mackey functor $\underline{\pi}_n Cof(u_{l(\xi-\xi^p)})$ is $\KK_p\bZp$. It remains to show that the spectrum $\Sigma^n H\KK_p\bZp$ has only $np$-slice, using the following Lemma \ref{Kpslice}.
\end{proof}

Note that we have the following equivalences:    

$$S^\xi \wedge H\KK_p\bZp \simeq \ast, S^{\xi^q} \wedge H\KK_p\bZp \simeq \ast \text{ and } S^{\xi^p} \wedge H\KK_p\bZp \simeq \Sigma^2 H\KK_p\bZp.$$

Therefore, for any $k \geq 0,$ we have $\Sigma^{k\rho_G} \wedge H\KK_p\bZp \simeq \Sigma^{kq} H\KK_p\bZp.$ Hence, by Proposition \ref{regular}, the only nontrivial slice of $\Sigma^{kq} H\KK_p\bZp$ is the $kpq$-slice $\Sigma^{kq} H\KK_p\bZp$. This gives the clue for the slices of $\Sigma^n H\KK_p\bZp$. 

\begin{lemma}\label{Kpslice}
The spectrum $\Sigma^{n} H\KK_p\bZp$ is a $pn$-slice.
\end{lemma}
\begin{proof}
If $\Sigma^{n} H\KK_p\bZp = P^t_t(X)$ for some $t$, then by the filtration of the homotopy groups we have $F^{s-1} \underline{\pi}_n \Sigma^n H\KK_p\bZp \neq F^s \underline{\pi}_n \Sigma^n H\KK_p\bZp$. Then, theorem \ref{algfil} yields $\mathcal{F}^{(s+n)/n} \KK_p\bZp \neq \mathcal{F}^{(s+n-1)/n} \KK_p\bZp.$ This can only happen when 
$$(s+n-1)/n < p \text{ and }(s+n)/n\geq p.$$ This gives $s= n(p-1).$ But note that $F^{s-1} \underline{\pi}_n \Sigma^n H\KK_p\bZp$ is empty and $F^{s} \underline{\pi}_n \Sigma^n H\KK_p\bZp \cong \KK_p \bZp.$ This gives $t-s=n$ and hence $t= np.$ Therefore, $\Sigma^{n} H\KK_p\bZp$ has only $np$-slice.
\end{proof}
\begin{prop}\label{res}
Let $X$ be a $G$-spectrum such that $\underline{\pi}_\bigstar(X)$ is cohomological. Then, $X$ is a $k$-slice if and only if both $res^G_{C_p}(X)$ and $res^G_{C_q}(X)$ are $k$-slices.
\end{prop}

\begin{proof}
Assume both $res^G_{C_p}(X)$ and $res^G_{C_q}(X)$ are $k$-slices. Then, by Proposition \ref{cohmac} $X$ is a $k$-slice for the group $G.$ The other direction follows from \cite[Proposition 4.13]{HHR16}.
\end{proof}

This Proposition suggests that computation of the $C_p$-slices is relevant to the calculations of the $C_{pq}$-slice for spectra which are modules over $H \uZ.$ An simple example of a $C_p$-$H\uZ$-module is the spectrum $S^W \wedge H\uZ,$ where $W$ is a representation of $C_p.$ Though the slices of this spectrum are known from \cite{Yar17}, we again compute the slices and obtain

\begin{lemma}\label{cpslice}
Let $W= m+n \xi$ be a representation of $C_p.$ Then the spectrum $\Sigma^W H \uZ$ is $C_p$-$dim(W)$-slice if and only if 

(i) $n \leq m \leq n+4$ for $p =3.$

(ii) $\frac{2n}{p-1} \leq m \leq \frac{2n+3p}{p-1}$ for $p \geq 5.$
\end{lemma}

\begin{proof}

For case (i), we have the regular representation, $\rho_{C_3} = 1 + \xi.$ Therefore, we have $S^W \wedge H\uZ \cong S^{n\rho_{C_3}} \wedge\Sigma^{m-n}H\uZ.$ So, the problem reduces to case when $\Sigma^{m-n}H\uZ$ is an $(m-n)$-slice.

It is immediate that $\Sigma^{m-n}H\uZ \geq m-n.$ By \cite[Theorem 4.3]{Lew88}, it follows that if $0 \leq m-n \leq 4$, then $\Sigma^{m-n}H\uZ \leq m-n.$

Next, consider case (ii): To show that $\Sigma^W H \uZ \geq dim(W)$, Theorem \ref{connective} implies that it is enough to compute the homotopy groups $\pi_k(\Phi^{C_p} \Sigma^W H\uZ)$ and $\pi_k(S^{dim(W)}\wedge H\uZ).$ It turns out that 

$$\pi_k(\Phi^{C_p} \Sigma^W H\uZ) =0 \text{ for } k< \frac{mp}{p}$$

and 
$$\pi_k(S^{dim(W)}\wedge H\uZ) =0 \text{ for } k < dim(W).$$ 

If $mp < dim(W),$ then $m=min \{ k : \underline{\pi}_k(S^{dim(W)}\wedge H\uZ) \neq 0\}$. Therefore, $\Sigma^W H\uZ$ is not in $dim(W)$-slice category. Therefore, $mp \geq dim(W)= m +2n$.

Next, we find conditions which ensure that such that $\Sigma^W H\uZ \leq dim(W).$ That is,  we must show that 

$$[S^{k\rho_{C_p}}, \Sigma^W H\uZ]^{C_p}\cong \tilde{H}^{W-k\rho_{C_p}}_{C_p}(S^0; \uZ)=0 \text{ for } k > \frac{dim(W)}{p}.$$
Following \cite[Theorem 4.3]{Lew88}, this group is non zero if and only if $m -k \geq 3$. Therefore, there exists some $k$ such that the group $\tilde{H}^{W-k\rho_{C_p}}_{C_p}(S^0; \uZ)$ vanishes if and only if $m > \frac{2n +3p}{p-1}.$ This gives $S^W \wedge H\uZ \leq dim(W)$ if and only if for all $k \geq 0,$ $m \leq \frac{2n +3p}{p-1}.$ Hence the result follows.   
\end{proof}

\begin{remark}\label{rmkyar}
For any $C_p$-spectrum $S^W \wedge H\uZ$, we have either $S^W \wedge H \uZ \leq |W|$ or $\geq |W|.$ In any case, in \cite{Yar17}, Yarnall proved that $S^W \wedge H\uZ$ has a spherical $dim(V)$-slice.
\end{remark}
\begin{cor}
Let $V$ be a $C_{pq}$-representation such that $V = a+ b \xi+ c\xi^p +d\xi^q.$ Then the spectrum $\Sigma^V H\uZ$ is a $C_{pq}$-dim($V$)-slice if  

i) $b+d \leq a+2c \leq b+d+4$ for $p=3$ or $\frac{2(b+d)}{p-1} \leq a+2c \leq \frac{2(b+d)+3p}{p-1}$ for $p\geq 5.$

ii) $\frac{2(b+c)}{q-1} \leq a+2d \leq \frac{2(b+c)+3q}{q-1}.$ 
\end{cor}

\begin{proof}
If $\Sigma^V H\uZ$ is slice, it must be a $dim(V)$-slice. Using Proposition \ref{res}, it is enough to show that  both the spectra $res^G_{C_p}(\Sigma^V H \uZ)$ and $res^G_{C_q}(\Sigma^V H \uZ)$ are $dim(V)$-slices. Hence, the result follows from Lemma \ref{cpslice}.
\end{proof}

\begin{prop}\label{slice}
For the $C_{pq}$-spectrum $S^V \wedge H\uZ$ we have 

(i) if $res^G_{C_p}(S^V \wedge H \uZ) \leq dim(V)$ and $res^G_{C_q}(S^V \wedge H \uZ) \leq dim(V),$ then it has a spherical dim(V)-slice Also, for $0 \leq k \leq dim(V),$ the $k$-slices of $S^V \wedge H\uZ$ are given by the certain wedges of the suspensions of $H \KK_p\bZp$ or $H\KK_q \bZq$ or the both.

(ii) if $res^G_{C_p}(S^V \wedge H \uZ) \geq  dim(V)$ and $res^G_{C_q}(S^V \wedge H \uZ)\geq dim(V),$ then it has a spherical $dim(V)$-slice. 

(iii) if $res^G_{C_p}(S^V \wedge H \uZ) \geq dim(V)$ and $res^G_{C_q}(S^V \wedge H \uZ) \leq dim(V),$ then it has a spherical $dim(V)$-slice. 

(iv) if $res^G_{C_p}(S^V \wedge H \uZ) \leq dim(V)$ and $res^G_{C_q}(S^V \wedge H \uZ) \geq dim(V),$ then it has a spherical $dim(V)$-slice. 

\end{prop}

\begin{proof}

Assume that the $C_{pq}$-representation is of the form $a+b\xi+c\xi^q +d\xi^q$ such that $a, b, c$ and $d$ are non-negative integers. Now consider the $C_p$-spectrum $res^G_{C_p}(S^V \wedge H \uZ)$ and the $C_q$-spectrum $res^G_{C_q}(S^V \wedge H \uZ).$

For case (i), the hypothesis immediately gives 

$$a+2c \leq dim(V)/p \text{ and } a+2d \leq dim(V)/q.$$ Therefore, from \cite{Yar17}, there exists $C_p$ and $C_q$-representations $m+n \xi$ and $m^\prime + n^\prime \xi$ respectively such that $S^{m+n \xi}\wedge H\uZ$ is a $dim(V)$-slice of  $res^G_{C_p}(S^V \wedge H \uZ)$ and $S^{m^\prime +n^\prime \xi}\wedge H\uZ$ is a $dim(V)$-slice of $res^G_{C_q}(S^V \wedge H \uZ)$.
 
 Next, we consider the representation 
 
 $$W = V -(b+d-n)(\xi-\xi^p) - (b+c - n^\prime)(\xi-\xi^q).$$ 
 
 Then, it is readily follows that the spectrum $res^{G}_H S^W \wedge H\uZ$ is a $H$-$dim(V)$-slice for both $H= C_p$ and $C_q.$ Hence, Proposition \ref{res} implies that $S^W \wedge H\uZ$ is a $dim(V)$-slice for the group $C_{pq}.$ Also, note that the restrictions of the cofiber of the composition 
 $$u_{(b+d-n)(\xi-\xi^p)} \circ u_{(b+c - n^\prime)(\xi-\xi^q)}: S^W \wedge H\uZ \to S^V \wedge H\uZ$$
  has slices with filtration $< dim(V).$
  
  For case (ii), we have a similar argument: consider the map 
  $$ u_{(b+d-n)(\xi-\xi^p)} \circ u_{(b+c - n^\prime)(\xi-\xi^q)}: S^V \wedge H\uZ \to S^{V+(b+d-n)(\xi-\xi^p) +(b+c - n^\prime)(\xi-\xi^q)} \wedge H\uZ.$$
  The restriction of its fiber to the subgroups $C_p$ and $C_q$ has slices less than $dim(V).$ Also, the spectrum $ S^{V+(b+d-n)(\xi-\xi^p) +(b+c - n^\prime)(\xi-\xi^q)} \wedge H\uZ$ is a $C_{pq}$-$dim(V)$-slice for $S^V \wedge H\uZ.$ 
  
  Finally, we consider case (iii). By hypothesis  $res^G_{C_p}(S^V \wedge H\uZ) \geq dim(V).$ Choose $l = \lceil \frac{(a+2c)(p-1)+2(b+d)-3p}{2p} \rceil.$ Then the spectrum $res^{G}_{C_p}(S^{V+l (\xi-\xi^p)}\wedge H\uZ) \leq dim(V).$ Also, there is a map 
  
  $$u_{l(\xi-\xi^p)}:S^V \to S^{V+ l (\xi-\xi^p)}\wedge H\uZ$$ whose fiber is a  wedge of trivial suspensions of $H \KK_p\bZp$, $H\KK_q\bZq$, or both and they are  $\geq dim(V).$ The spectrum $S^{V+l(\xi-\xi^p)}\leq dim(V)$ as its restriction to the subgroups $C_p$ and $C_q$ are $\leq dim(V)$. So, it has slices if the slice filtration is less equal to $dim(V).$ Hence, we get the slice tower for $S^V \wedge H\uZ.$

\end{proof}

\begin{thm}\label{main2}
For $\alpha \in RO(C_{pq})$  there exists $\beta \in RO(C_{pq})$ such that  $S^\alpha \wedge H\uZ$ has a $dim(\alpha)$-slice $S^\beta \wedge H\uZ.$ The other slices are suspensions of $H \KK_p \bZp$ or $H\KK_q \bZq$ or wedges of them.
\end{thm}
\begin{proof}
We can always find some $k \in \Z$ such that $\alpha + k \rho_{G}$ is honest representation  of $G$. Therefore, using Proposition \ref{regular} it is enough to consider $\alpha =V= a+b\xi+c\xi^p +d \xi^q$ for $a, b,c, \text{and }d$ are non-negative integers. Then the Propositions \ref{cofiber} and \ref{slice} together imply the result.
\end{proof}

Next, we consider few examples for the group $G=C_{15}$, so, $p=3$ and $q=5.$

\begin{exam}
We'll give a slice tower for the $C_{15}$-spectrum $X = S^{11\xi^5}\wedge H\uZ.$ Note that $Res^{G}_{C_3}(X) \cong S^{11\xi} \wedge H\uZ$ and $Res^{G}_{C_5}(X) \cong S^{22} \wedge H\uZ.$ Therefore, we are in case where $Res^{G}_{C_3}(X) \leq 22$ and $Res^{G}_{C_5}(X) \geq 22.$

Choose $l$ to be the lease positive integer such that the spectrum $res^G_{C_5}S^{l(\xi-\xi^5})\wedge X \leq 22.$ Then, $l$ should be equal to $\lceil \frac{22.(5-1)- 3.5}{2.5} \rceil  =8.$ Therefore, until   the $22$-slice all the higher dimension slices are obtain by the computations of the fiber of the map $u_{\xi-\xi^q}.$ Therefore, iterative use of the Proposition \ref{slice} we obtain the slice tower as follows:

$$\xymatrix{95\text{-slice:}& \Sigma^{19} H\KK_q\bZq \ar[r] & S^{11\xi^q} \wedge H\uZ \ar[d]^{u_{\xi-\xi^q}} \\
85\text{-slice:}& \Sigma^{17} H\KK_q\bZq \ar[r] & S^{\xi +10\xi^q} \wedge H\uZ \ar[d]^{u_{\xi-\xi^q}} \\
75\text{-slice:}&\Sigma^{15} H\KK_q\bZq \ar[r] & S^{2\xi + 9 \xi^q} \wedge H\uZ \ar[d]^{u_{\xi-\xi^q}}\\ 
65\text{-slice:}&\Sigma^{13} H\KK_q\bZq \ar[r] & S^{3\xi + 8\xi^q} \wedge H\uZ \ar[d]^{u_{\xi-\xi^q}} \\ 
55\text{-slice:}& \Sigma^{11} H\KK_q\bZq \ar[r] & S^{4\xi + 7\xi^q} \wedge H\uZ \ar[d]^{u_{\xi-\xi^q}}\\
45\text{-slice:}& \Sigma^{9} H\KK_q\bZq \ar[r] & S^{5\xi + 6\xi^q} \wedge H\uZ \ar[d]^{u_{\xi-\xi^q}}\\
35\text{-slice:}& \Sigma^{7} H\KK_q\bZq \ar[r] & S^{6\xi +5\xi^q} \wedge H\uZ \ar[d]^{u_{\xi-\xi^q}}\\
25\text{-slice:}& \Sigma^{5} H\KK_q\bZq \ar[r] & S^{7\xi+4\xi^q} \wedge H\uZ \ar[d]^{u_{\xi-\xi^q}}\\
22 \text{-slice:}& S^{7\xi+\xi^p +3\xi^q}\wedge H\uZ \ar[r] & S^{8\xi+3\xi^q} \wedge H\uZ \ar[d]\\ 0 \text{-slice:}&
& H\KK_p \bZp}$$
\end{exam}

\begin{exam}First we construct the slice tower for the spectrum $S^6 \wedge H\uZ.$ The restrictions satisfy 
$$res^G_{C_p} (S^6 \wedge H\uZ) \geq 6$$
and
$$res^G_{C_q} (S^6 \wedge H\uZ) \geq 6.$$

Therefore, we are case (ii) of the Proposition \ref{slice}. Using the technique mentioned in the proof we get the following slice tower:
$$\xymatrix{15-slice: & \Sigma^3 \KK_q\bZq \ar[r] & S^6 \wedge H\uZ \ar[d]^{u_{\xi-\xi^q}} \\
9-slice: & \Sigma^3 \KK_p\bZp \ar[r] & S^{6+\xi-\xi^q}\wedge H\uZ \ar[d]^{u_{\xi-\xi^p}} \\ 6-slice: & & S^{6+2\xi-\xi^p-\xi^q}\wedge H\uZ }$$
\end{exam}

\bibliographystyle{siam}
\bibliography{algtop}{}

\mbox{ }\\
\end{document}